\documentclass[12pt]{amsart}
\usepackage{fullpage}
\usepackage{enumitem}
\usepackage{hyperref}
\usepackage{refcount}

\newtheorem{theorem}{Theorem}
\newtheorem{lemma}[theorem]{Lemma}
\newtheorem{conjecture}[theorem]{Conjecture}

\newcounter{claimcounter}
\numberwithin{claimcounter}{theorem}
\newtheorem{claim}[claimcounter]{Claim}

\DeclareMathOperator{\Prob}{\mathbb{P}}
\DeclareMathOperator{\Exp}{\mathbb{E}}

\newcommand{\floor}[1]{\left\lfloor #1 \right\rfloor}
\newcommand{\ceiling}[1]{\left\lceil #1 \right\rceil}
\title{Long rainbow cycles and Hamiltonian cycles using many colors in properly edge-colored complete graphs}
\author{J\'ozsef Balogh}
\address{
  Department of Mathematical Sciences\\ 
  University of Illinois at Urbana-Champaign\\ 
  Urbana, Illinois 61801, USA.
}
\email[J\'ozsef Balogh]{jobal@math.uiuc.edu} 
\thanks{
  The research of the first author is partially supported by 
  NSF Grant DMS-1500121, 
  the Arnold O. Beckman Research Award (UIUC Campus Research Board 15006)
  and the Langan Scholor Fund (UIUC). 
  Work was done while the first author was a Visiting Fellow Commoner at Trinity College, Cambridge
}
\author{Theodore Molla}
\email[Theodore Molla]{molla@illinois.edu}
\thanks{
  The research of the second author is partially supported by NSF Grant DMS-1500121.
} 
\date{\today}

\begin{document}
\begin{abstract}
  We prove two results regarding cycles in properly edge-colored graphs.
  First, we make a small improvement to the recent breakthrough work of Alon, Pokrovskiy and Sudakov
  who showed that every properly edge-colored complete graph $G$ on $n$ vertices has a rainbow cycle on at least $n - O(n^{3/4})$ 
  vertices, by showing that $G$ has a rainbow cycle on at least  $n - O(\log n \sqrt{n})$ vertices. 
  Second,
  by modifying the argument of Hatami and Shor which gives a lower bound for the length of a partial transversal in a Latin Square, 
  we prove that every properly colored complete graph has a Hamilton cycle in which
  at least $n - O((\log n)^2)$ different colors appear.
  For large $n$, this is an improvement of the previous best known lower bound of $n - \sqrt{2n}$ of Andersen. 
\end{abstract}
\maketitle

\section{Introduction}\label{introduction}

Let $G$ be a graph.  An edge-coloring 
is a \textit{proper edge-coloring} if the set of edges incident to a vertex are given distinct colors.
If $G$ is edge-colored, we say that $H \subset G$ is \textit{rainbow} if 
the colors assigned to the edges of $H$ are distinct.

There has been extensive research on rainbow properties of properly edge-colored graphs.
One problem that has seen significant recent interest is to find many edge-disjoint rainbow spanning trees, 
see Akbari \& Alipour \cite{aa}, Balogh, Liu \& Montgomery \cite{blm},  Pokrovskiy \& Sudakov \cite{ps}, 
Carraher, Hartke \& Horn \cite{chh}, and Horn \& Nelson \cite{hn}.
A  related old  conjecture of Andersen \cite{A} is the following, which if true would be best possible.
\begin{conjecture}\label{andersen_c}
  Every properly edge-colored complete graph on $n$ vertices contains a rainbow path on $n-1$ vertices.
\end{conjecture}
In the same paper, Andersen proved the following result.
\begin{theorem}\label{andersen_t}
  Every properly edge-colored complete graph on $n$ vertices
  contains a Hamiltonian cycle in which at least $n - \sqrt{2 n}$ distinct colors appear.
\end{theorem}
In this paper, we make the following improvement (for large $n$) to Theorem~\ref{andersen_t}.
\begin{theorem}\label{rpf}
  There exists a constant $C$ such that for every $n$ the following holds.
  Every properly edge-colored complete graph on $n$ vertices
  contains a Hamiltonian cycle in which at least $n - C(\log n)^2$ distinct colors appear.
\end{theorem}
Note that our proof of Theorem~\ref{rpf} very closely follows the proof of Hatami \& Shor's \cite{hs}
result on the length of a partial traversal in a Latin square.

Instead of asking for a Hamiltonian cycle that uses many colors, another way to approach Conjecture~\ref{andersen_c},
is to attempt to find a long rainbow cycle.  This problem has received recent interest. 
 Akbari, Etesami, Mahini \& Mahmoody \cite{aemm} proved that a cycle of length $n/2- 1$ exists in every
properly colored complete graph $G$ on $n$ vertices.
Then  Gy\'arf\'as \& Mhalla \cite{gym} proved that a rainbow path of length $(2n + 1)/3$ exists in $G$ provided that, 
for every color $\alpha$ used, the set of edges given the color $\alpha$ forms a perfect matching in $G$.
Not much later  Gy\'arf\'as, Ruszink\'o, S\'ark\"ozy, \& Schelp \cite{grss} showed that $G$ contains a 
rainbow cycle of length $(4/7 + o(1))n$ for every proper edge-coloring. 
Then, independently, both Gebauer \& Mousset \cite{gm} and Chen \& Li \cite{cl}
showed that a path of length $(3/4 -  (1))n$ exists when $G$ is properly colored.
This was the best known lower bound until very recently
Alon, Pokrovskiy \& Sudakov \cite{aps} established the following breakthrough result.
\begin{theorem}\label{aps_t}
   For all sufficiently large $n$, every properly edge-colored complete graph on $n$ vertices contains a rainbow cycle
   on at least $n - 24n^{3/4}$ vertices.
\end{theorem}
Heavily relying on the methods developed in \cite{aps}, we make the following improvement to Theorem~\ref{aps_t}.
\begin{theorem}\label{rc}
  There exists a constant $C$ such that the following holds.
  If $G$ is a properly edge-colored complete graph on $n$ vertices for  $n$  sufficiently large, 
  then there exists a rainbow cycle in $G$ on at least $n - C\log n\sqrt n$ vertices.
\end{theorem}
To prove Theorem~\ref{rc}, we use the following theorem which is an extension of Theorem~1.3 in \cite{aps}.
Our improvement, and one of the main observations that drives our proof, 
is that essentially the same argument as the one given in \cite{aps} works when the sizes of the two sets are unbalanced,
i.e.,\ only one of the two sets needs to have order $\Omega\left(\left(\log n/p\right)^2\right)$, 
the other can be as small as $\Omega\left(\log n/p\right)$.
\begin{theorem}\label{random_selection}
  For every sufficiently small $\varepsilon > 0$ 
  there exists a constant $C$ such that the following holds.
  Let $G$ be a properly edge-colored graph on $n$ vertices 
  such that $\delta(G) \ge (1 - \delta)n$ for some $\delta = \delta(n)$.
  Let $H$ be the spanning subgraph obtained by choosing
  every color class independently at random with probability $p$. 
  Then the following holds with high probability.
  \begin{enumerate}[label=(\alph*), noitemsep] 
    \item If $(1 - \delta)n \ge C \frac{\log n}{p}$, then all vertices $v$ have degree $(1 \pm \varepsilon)p \cdot d_G(v)$ in $H$.
    \item For every pair $A$ and $B$ of disjoint vertex sets, if 
      $|A| \ge C\frac{\log n}{p}$ and \\ $|B| \ge \max\{C(\frac{\log n}{p})^2,C \delta n\}$, then
      $e_H(A, B) \ge (1 - \varepsilon)p|A||B|$.
  \end{enumerate}
\end{theorem}

\subsection{Definitions and notation}

Most of our notation is standard except possibly the following.
Let $G$ be a graph and let $A$ and $B$ be disjoint vertex subsets.
We let 
\begin{equation*}
  E_G(A,B) = \{xy \in E(G) : x \in A \text{ and } y \in B\},
\end{equation*} 
and let $e_G(A,B) = |E_G(A,B)|$.  For a vertex subset $U$, we let 
\begin{equation*}
  N_G(U) = \bigcup_{u \in U} N_G(u).
\end{equation*}

For a path $P = v_1,\dotsc,v_m$  we call $v_1$ and $v_m$ the \textit{endpoints} of $P$ 
and we call a path $Q$ a \textit{subpath} of $P$ if
$Q = v_i,\dotsc,v_j$ for some $1 \le i \le j \le m$.
For a set $X$ and $0 \le y \le |X|$, we let $\binom{X}{y} = \{Y \subseteq X : |Y| = y\}$ be the set of all
subsets of $X$ that have order exactly $y$.
All logarithms are base $2$ unless otherwise specified.

\section{Proof of Theorem~\ref{random_selection}}

We use the following form of the well-known Chernoff bound. 
\begin{lemma}[Chernoff bound]\label{chernoff}
  Let $X$ be a binomial random variable with parameters $(n,p)$. 
  Then for every $\varepsilon \in (0,1)$ we have that
  \begin{equation*}
    \Prob(|X - pn| \ge \varepsilon pn) \le 2e^{-pn \varepsilon^2/3}.
  \end{equation*}
\end{lemma}

For an $\varepsilon > 0$, call a pair of disjoint vertex subsets $A$, $B$ of a properly edge-colored graph $G$
\textit{$\varepsilon$-nearly-rainbow} if there are $(1 - \varepsilon)|A||B|$ different colors that
appears on the edges $E_G(A,B)$.

The following is essentially equivalent to Lemma~2.2 in \cite{aps} and is a simple application of the 
Chernoff bound (Lemma~\ref{chernoff}).
\begin{lemma}\label{nearly_rainbow}
  For every $\varepsilon > 0$, there exists a constant $C$ such that the following holds.
  Let $G$ be a properly edge-colored graph on $n$ vertices, and
  let $H$ be the spanning subgraph of $G$ obtained by choosing
  every color class independently at random with probability $p$.
  Then, with high probability,
  for every $\varepsilon$-nearly rainbow pair $S,T$ such that $|S| = |T| \ge C \frac{\log n}{p}$ we have that
  \begin{equation*}
    e_H(S, T) \ge (1 - 2\varepsilon)p|S||T|.
  \end{equation*}
\end{lemma}

The proof of the following lemma is very similar to Lemma~2.3 in \cite{aps}.  
The main differences are that it can be applied to graphs that are not complete and 
that the sets $A$ and $B$ can be of different sizes.
\begin{lemma}\label{partition}
  For every sufficiently small $\varepsilon > 0$, there exists $C$ such that when $n$ is sufficiently large the following holds.
  Let $G$ be a graph on $n$ vertices and let $A$ and $B$ be two disjoint vertex subsets of size $a$ and $b$, respectively, with $a\le b$.
  Suppose $\delta(G) \ge (1 - \delta)n$ for some $\delta=\delta(n) > 0$.
  If $y$ divides both $a$ and $b$, and $b \ge \max\{C y^2, C \delta n\}$ and $y \ge C$, then
  there exists a partition $\{A_i\}$ of $A$ into parts of size $y$ 
  and a partition $\{B_j\}$ of $B$ into parts of size $y$ such that
  all but $\varepsilon \cdot \frac{ab}{y^2}$ of the pairs $A_i,B_j$ are $\varepsilon$-nearly-rainbow.
\end{lemma}
\begin{proof}
  Let $\varepsilon > 0$ be sufficiently small. 
  We choose $C \ge 2\varepsilon^{-2}$ so
  we have that ${y^2}/{b} \le 1/C\le \varepsilon^2/2$ and ${\delta n}/{b} \le   1/C\le  \varepsilon^2/2$.
  
  Let 
  \begin{equation*}
    \Omega = \{\alpha : \text{$\exists e \in E(A,B)$ such that $e$ is colored with $\alpha$} \}
  \end{equation*}
  be the set of colors used on the edges of $E(A,B)$.
  Assume $S$ and $T$ are selected uniformly at random from $\binom{A}{y}$ and $\binom{B}{y}$, respectively.

  For every distinct $e, e' \in E(S,T)$ that do not share an endpoint 
  \begin{equation}\label{PST}
    \Prob(e \in E(S,T)) = \frac{y^2}{ab} \qquad \text{ and } \qquad 
    \Prob(e,e' \in E(S,T)) = \frac{y^2(y-1)^2}{ab(a-1)(b - 1)}.
  \end{equation}
  Let $\alpha \in \Omega$ and  $E_{\alpha} = \{e \in E(A,B) : \text{$e$ is given color $\alpha$}\}$.
  Since the edge-coloring is proper, no two edges in $E_{\alpha}$ share an endpoint, which
  implies that $|E_{\alpha}| \le \min\{a,b\} = a$.
  Therefore, using inclusion-exclusion and \eqref{PST}, we get the following lower bound on the
  probability that the color $\alpha$ is used on an edge in $E(S,T)$:
  \begin{align*}
    \Prob(|E_{\alpha} \cap E(S,T)| \ge 1) &\ge 
    \sum_{e \in E_{\alpha}} \Prob(e \in E(S,T)) - 
    \sum_{ \{e, e'\} \in \binom{E_{\alpha}}{2} } \Prob(e,e' \in E(S,T))  \\
    &= \frac{y^2}{ab} |E_{\alpha}| - \frac{y^2(y-1)^2}{ab(a-1)(b-1)} \binom{|E_{\alpha}|}{2} 
    = \left(1 - \frac{(y-1)^2(|E_{\alpha}| - 1)}{2(a-1)(b-1)}\right)\frac{y^2}{ab}|E_{\alpha}| \\
    &\ge \left(1 - \frac{(y-1)^2}{2(b-1)} \right)\frac{y^2}{ab}|E_{\alpha}| 
    \ge \left(1 - \frac{\varepsilon^2}{2} \right)\frac{y^2}{ab}|E_{\alpha}|.
  \end{align*}

  Let $Z$ be the number of different colors used on the edges of $E(S,T)$.
  Using linearity of expectation and the fact that 
  \begin{equation*}
    \sum_{\alpha \in \Omega} |E_\alpha| = |\cup_{\alpha \in \Omega} E_{\alpha}| = |E(A,B)| \ge ab - a \delta n
  \end{equation*}
  we have that
  \begin{equation*}
    \Exp(Z) \ge \sum_{\alpha \in \Omega} (1 - \varepsilon^2/2) \frac{y^2}{ab} |E_{\alpha}| \ge 
    (1 - \varepsilon^2/2) \frac{y^2}{ab} (ab - a \delta n) \ge 
    \left(1 - \frac{\varepsilon^2}{2} - \frac{\delta n}{b}\right) y^2 \ge
    (1 - \varepsilon^2) y^2.
  \end{equation*}
  Clearly, $Z \le |E(S,T)| \le y^2$, so $y^2 - Z \ge 0$, and $\Exp(y^2 - Z) \le \varepsilon^2 y^2$.
  Markov's inequality then implies that $\Prob(y^2 - Z \ge \varepsilon y^2) \le \varepsilon$, so the 
  probability that $S,T$ is $\varepsilon$-nearly rainbow is at least $1 - \varepsilon$.

  Select a partition
  $\{A_i\}$ of $A$ into parts of size $y$ and 
  a partition $\{B_j\}$ of $B$ into parts of size $y$ uniformly at random.
  The expected fraction of the $\frac{ab}{y^2}$ pairs of sets $A_i$ and $B_j$ that
  are not $\varepsilon$-nearly rainbow is at most $\varepsilon$, so there exists a partition 
  such that all but at at most $\varepsilon \frac{ab}{y^2}$ pairs of sets are $\varepsilon$-nearly rainbow.
\end{proof}

\begin{proof}[Proof of Theorem~\ref{random_selection}]
  Assume $\varepsilon > 0$ is sufficiently small, and 
  that $C \ge 6\varepsilon^{-2}$ is large enough so that both
  Lemmas~\ref{nearly_rainbow}~and~\ref{partition} hold.
  We will show that the conditions of Theorem~\ref{random_selection} hold 
  with $C' = 4 C^3 \varepsilon^{-1}$ and $5 \varepsilon$ playing
  the roles of $C$ and $\varepsilon$, respectively.

  Let $v \in V(G)$.  
  Because $G$ is properly edge-colored, for every $v \in V(G)$, 
  the edges incident to $v$ are rainbow, so
  the number of edges incident to $v$ that are in $H$ is 
  binomial distributed with parameters $(d_G(v), p)$. 
  Because $p \cdot d_G(v) \ge p(1 - \delta)n \ge C \log n$, we have that
  \begin{equation*}
    n \cdot 2e^{-p \cdot d_G(v)\varepsilon^2/3} \le 2 n^{-1}, 
  \end{equation*}
  so the Chernoff bound (Lemma~\ref{chernoff}) and the union bound imply that 
  \begin{equation*}
    d_H(v) = (1 \pm \varepsilon)pd_G(v)
  \end{equation*} 
  for every $v \in V(G)$ with high probability.

  Fix $y = \ceiling{C \log n/p}$.
  By Lemma~\ref{nearly_rainbow}, with high probability, for every $\varepsilon$-nearly rainbow pair $S,T$ such that
  $|S| = |T| = y$, we have that
  \begin{equation}\label{ST}
    e_H(S,T) \ge (1 - 2\varepsilon)y^2.
  \end{equation}

  Now fix $a = y$ and let $b$ be the smallest number 
  larger than $\max\{C y^2, C \delta n\}$ that is divisible by $y$. 
  Let $A$ and $B$ be vertex disjoint subsets of orders $a$ and $b$, respectively.
  Lemma~\ref{partition} implies that there exists 
  a partition $\{B_j\}$ of $B$ into parts of size $y$ such that all but an $\varepsilon$ fraction of the
  pairs $A,B_j$ are $\varepsilon$-nearly-rainbow, 
  i.e.,\ if we let
  $J = \{ j : \text{$A,B_j$ is $\varepsilon$-nearly-rainbow} \}$, $|J| \ge (1 - \varepsilon)ab/y^2 = (1 - \varepsilon)b/y$.
  Therefore, with \eqref{ST}, we have that with high probability
  \begin{equation}\label{AB}
    e_H(A, B) \ge \sum_{j \in J} e_H(A, B_j) \ge (1 - \varepsilon)ab/y^2 \cdot (1 - 2\varepsilon)y^2 = (1 - 3\varepsilon) ab.
  \end{equation}

  Finally, to complete the proof, let $A$ and $B$ be disjoint vertex subsets 
  such that $|A| \ge C' \log n/p$, $|B| \ge  \max\{C' (\log n/p)^2, C' \delta n\}$.
  Note that 
  $|A| \ge \varepsilon^{-1} a$ and $|B| \ge \varepsilon^{-1} b$, so
  there are at least $(1 - \varepsilon)|A|/a$ disjoint subsets of $A$ of size $a$ and
  at least $(1 - \varepsilon)|B|/b$ disjoint subsets of $B$ of size $b$.
  Hence, \eqref{AB} implies that with high probability
  \begin{equation*}
    e_H(A, B) \ge (1 - \varepsilon)|A|/a \cdot (1 - \varepsilon)|B|/b \cdot (1 - 3\varepsilon) ab = (1 - 5\varepsilon) |A||B|. \qedhere
  \end{equation*}
\end{proof}

\section{Long rainbow cycles}

This appears as Lemma~3.1 in \cite{aps}.
\begin{lemma}\label{long_path_forest}
  For all $\gamma, \delta, n$ with $\delta \ge \gamma$ and $3\gamma\delta - \gamma^2/2 > n^{-1}$ the following holds.
  Let $G$ be a properly edge-colored graph on $n$ vertices such that $\delta(G) \ge (1 - \delta)n$.  
  Then $G$ contains a rainbow path forest 
  $\mathcal{P}$ with at most $\gamma n$ paths and $|E(\mathcal{P})| \ge (1 - 4\delta)n$.
\end{lemma}

For $0 < a \le b$, call a graph $H$ an \textit{$(a,b)$-expander}, if the following holds:
\begin{enumerate}[label=(E\arabic*), noitemsep]
  \item\label{E1} $\delta(H) \ge a$;
  \item\label{E2} if $A \subseteq V(H)$ such that $|A| \ge a$, then $|N_H(A)| \ge n - a - b$; and
  \item\label{E3} if $A$ and $B$ are disjoint subsets of order $a$ and $b$, respectively, then 
    $E_H(A, B) \neq \emptyset$.
\end{enumerate}
Note that \ref{E3} implies \ref{E2}, 
because \ref{E3} implies $|V(G) \setminus (N_H(A') \cup A')| \le b$ for every $A' \in \binom{A}{a}$,
but it is more convenient to state \ref{E2} separately.

\begin{lemma}
  \label{path_builder}
  Let $0 < a \le b \le n/4$, $r > 0$, and
  let $G$, $H_1$, $H_2$ and $H_3$, be edge-disjoint spanning subgraphs of the complete graph on $n$ vertices
  whose edges are edge-colored by pairwise disjoint sets of colors,
  such that $H_1$, $H_2$ and $H_3$ are each $(a,b)$-expanders.
  Suppose $\mathcal{P} = \{P_1, \dotsc, P_r\}$ is a rainbow path forest in $G$ such that for
   $U = V(G) \setminus V(\mathcal{P})$ we have  $|U| \le b$. 
  If $|P_1| \le n - a - |U|$, then there exists $e_j \in E(H_j)$ for $j \in [3]$, and 
  $i \in [r]$, such that there are two disjoint paths $P_1'$ and $P_i'$ in the graph induced in $G$ by $V(P_1) \cup V(P_i) \cup U$
  with the additional edges $\{e_1, e_2, e_3\}$ where 
  \begin{enumerate}
    \item\label{PB1} $P_i'$ is a subpath of $P_i$ on less than $|P_i|/2$ vertices
      (we allow $P_i'$ to be the path without vertices here), 
    \item\label{PB2} $|P_1'| \ge |P_1| + |P_i| - |P_i'|$, and 
    \item\label{PB3} $\mathcal{P}' = \mathcal{P} - P_1 - P_i \cup \{P_1' , P_i'\}$ is a rainbow path forest.
  \end{enumerate}
\end{lemma}
\begin{proof}
  Assume $|P_1| \le n - a - |U|$ and that the statement of the lemma does not hold. 
  Let $v_1, \dotsc, v_m$ be the vertices of $P_1$ in the order they appear on the path,
  let $T$ be the set of vertices on the paths $P_2, \dotsc, P_{r}$, and
  recall that $U = V(G) \setminus V(\mathcal{P})$.
  We have that 
  \begin{equation}
          |T| = n - |P_1| - |U| \ge a.  \label{T}
  \end{equation}

  \begin{claim}\label{P}
    Let $\sigma$ be a permutation of $\{1,2, 3\}$.  
    For every path $P$ such that
    \begin{itemize}
      \item $|P| \ge |P_1|$,
      \item $V(P) \subseteq V(P_1) \cup U$, and 
      \item $E(P) \subseteq E(P_1) \cup \{e_{\sigma(1)}, e_{\sigma(2)}\}$ where $e_\sigma(j) \in E(H_{\sigma(j)})$ for $j \in [2]$, 
    \end{itemize}
    there are no edges in $E(H_{\sigma(3)})$ incident to an endpoint of $P$ and a vertex in $T$.
  \end{claim}
  \begin{proof}
    Suppose, for a contradiction, that $v$ is an endpoint of such a path $P$ and 
    there exists $x \in N_{H_{\sigma(3)}}(v) \cap T$.
    Let $P_i \in \{P_2, \dotsc, P_r\}$ be the path containing $x$. 
    We can construct $P'_1$ by combining $P$ with the longer of the two subpaths in $P_i$ that have 
    $x$ as an endpoint.  By letting $P'_i$ be subpath of $P_i$ with the vertex set $V(P_i) \setminus V(P'_1)$ 
    (recall that the statement of the theorem allows $P'_i$ to be the path without vertices), 
    we have two paths $P'_1$ and $P'_i$ that satisfy the conditions of the lemma.
  \end{proof}

  Let $\mathcal{P}$ be the set of all paths that do not contain an edge from $E(H_3)$ and
  that meet the conditions of Claim~\ref{P}.
  Let 
  \begin{equation*}
    X = \{x \in V(G) : \text{there exists $P \in \mathcal{P}$ such that $x$ is an endpoint of $P$}\}.
  \end{equation*}
  Claim~\ref{P} implies that $E_{H_3}(X, T) = \emptyset$, so, with \eqref{T}, 
  we will contradict \ref{E3} and prove the lemma
  if we can show that $|X| \ge b$.

  Claim~\ref{P} implies that, in $H_1$, $v_1$ does not have a neighbor in $T$, so if we let,
  \begin{equation*}
    Y = (N_{H_1}(v_1) \cap U) \cup \{ v_j : v_{j+1} \in N_{H_1}(v_1) \cap V(P_1)\},
  \end{equation*}
  \ref{E1} implies that $|Y| = |N_{H_1}(v_1)| \ge a$.  
  Therefore, by \ref{E2},
  \begin{equation}\label{YH_2}
    |N_{H_2}(Y)| \ge n - a - b \ge 2b.
  \end{equation}
  We also have that $Y \subseteq X$.
  To see this, observe that
  if $y \in Y \cap U$, then $y,v_1, \dotsc, v_m$ is in $\mathcal{P}$, and
  if $y \in Y \setminus U$, then $y$ is on $P_1$, so $y = v_j$ for some $j \in [m]$ and the path 
  $v_j, \dotsc, v_1, v_{j+1}, \dotsc, v_m$ is in $\mathcal{P}$.
  
  We will now describe a mapping from $N_{H_2}(Y)$ to $X$ 
  such that, for every $x \in X$, at most two vertices in $N_{H_2}(Y)$ are mapped to $x$. 
  By \eqref{YH_2}, this will imply that $|X| \ge b$, which, as was previously mentioned, will prove the claim.
  To this end, let $v \in N_{H_2}(Y)$, and arbitrarily select some $y \in N_{H_2}(v) \cap Y$.
  Recall that there exists $P_y \in \mathcal{P}$ that has $y$ as an endpoint 
  and that does not contain edges from either $H_2$ or $H_3$.
  First assume that $v$ is not on $P_1$. 
  Then, by using $vy \in E(H_2)$, we can append $v$ to $P_y$ to create an element of $\mathcal{P}$
  with $v$ as an endpoint.
  Therefore, $v \in X$, so, in this case, we map $v$ to itself.
  Now assume that, $v = v_k$ for some $k \in [m]$, and
  recall that $Y$ does not intersect $T$, so $y \in V(P_1) \cup U$.
  If $y \in U$, then $v_k \neq v_1$, since $yv_1 \in E(H_1)$, so we can map $v_k$ to $v_{k-1}$, 
  because $v_{k-1} \dotsc v_1 y v_k \dotsc v_m$ is in $\mathcal{P}$.
  If $y \in P_1$, then $y = v_j$ for some $j \in [m]$. 
  Recall that by the definition of $Y$, $v_1v_{j+1}$ is an edge in $H_1$.
  If $k \ge j + 1$, 
  the path $v_{k-1}, \dotsc, v_{j+1},v_1, \dotsc, v_j, v_k, \dotsc, v_m$ is in $\mathcal{P}$,
  so we map $v_k$ to $v_{k-1}$.
  Similarly, if $k \le j - 1$, 
  the path $v_{k+1}, \dotsc, v_j,v_k, \dotsc, v_1, v_{j+1}, \dotsc, v_m$ is in $\mathcal{P}$, 
  so we map $v_k$ to $v_{k+1}$.
  Note that we have now proved the lemma, because for every every $x \in X$, at most two vertices in $N_{H_2}(Y)$ are mapped to $x$;
  if $x \in X \cap U$, then the only vertex that can be mapped to $x$ is $x$ itself, and
  if $x = v_k$ for some $v_k \in X \cap V(P_1)$ then $v_{k-1}$ and $v_{k+1}$ are the only vertices that can be mapped to $x$. 
\end{proof}

\edef\savetheoremcounter{\thetheorem}
\setcounterref{theorem}{rc}
\setcounter{claimcounter}{0}
\begin{proof}[Proof of Theorem~\ref{rc}]
  Assume $\varepsilon > 0$ is sufficiently small, and
  pick $C$ large enough so that, provided $n$ is sufficiently large, the conditions of 
  Theorem~\ref{random_selection} hold with $\varepsilon$ and $C$.
  Let $p = C \frac{\log n}{\sqrt{n}}$, $a = \sqrt{n}$, $b = n/4$, $\delta = 4p$, 
  and $\gamma = \frac{1}{C \log n \sqrt{n}}$.
  Let $G_1$ be a properly edge-colored complete graph on $n$ vertices.

  \begin{claim}\label{expander}
    There exist edge-disjoint spanning subgraphs, $G$, $H_1$, $H_2$ and $H_3$, of $G_1$ 
    that are properly edge-colored with pairwise disjoint sets of colors such that
    \begin{enumerate}[label=(\alph*), noitemsep]
      \item $\delta(G) \ge (1 - \delta)n$, and,
      \item for every $m \le \sqrt{n}$ and $i \in [3]$, if $m$ color classes are removed from $H_i$,
        then the resulting graph is an $(a,b)$-expander.
    \end{enumerate}
  \end{claim}
  \begin{proof}
    We form $H_1$ by selecting the color classes of $G_1$ randomly and independently with probability $p$.
    With high probability, 
    the conditions of Theorem~\ref{random_selection} hold. We fix such a subgraph $H_1$.

    Note that every vertex has degree $(1 \pm \varepsilon)p(n-1)$ in $H_1$, so
    if we let $G_2$ be the graph formed by removing the edges of $H_1$ from $G_1$, we
    have that every vertex has degree $(1 - p(1 \pm \varepsilon))(n-1)$ in $G_2$.
    Therefore, if we form $H_2$, by selecting the color classes of $G_2$ 
    randomly and independently with probability $p$, the conditions of Theorem~\ref{random_selection} 
    hold in $H_2$ with high probability.
    We fix such a graph $H_2$ and note that every vertex has degree $(1 \pm 1.1\varepsilon)p(n-1)$ in $H_2$.
    We then let $G_3$ be the graph formed by removing the edges of $H_2$ from $G_2$.
    Every vertex has degree $(1 - 2(p \pm 1.1\varepsilon))(n-1)$ in $G_3$, so if we form $H_3$ 
    by selecting the color classes of $G_3$
    randomly and independently with probability $p$, $H_3$ satisfies the conditions of
    Theorem~\ref{random_selection} with high probability, so we can fix such an $H_3$.
    We now have that for every $j \in [3]$, and every vertex $v$, 
    \begin{equation}\label{Hjdeg}
      d_{H_j}(v) \ge (1 \pm 1.2 \varepsilon)p(n-1) > 2 \sqrt{n}.
    \end{equation}
    Let $G$ be the graph formed by removing the edges of $H_3$ from $G_3$,
    and note that, with \eqref{Hjdeg}, for every vertex $v$ we have that
    \begin{equation*}
      d_{G}(v) = (n - 1) - \sum_{j\in [3]}d_{H_j}(v) \ge (n - 1) - 3(1 + 1.2 \varepsilon)np \ge (1 - \delta)n.
    \end{equation*}

    Because $a \ge C \frac{\log n}{p}$, $b \ge C \left(\frac{\log n}{p}\right)^2$ and $b \ge C \delta n$, 
    the conditions of Theorem~\ref{random_selection}, imply that, 
    for $j \in [3]$ and every pair of disjoint vertex sets $A$ and $B$ with sizes at least $a$ and $b$, respectively,
    $e_{H_j}(A,B) \ge (1 - \varepsilon) p|A||B|$.
    
    For each $j \in [3]$, form $H_j'$ by removing an arbitrary set of $m \le \sqrt{n}$ color classes from $H_j$.
    By \eqref{Hjdeg}, and the fact that $H_j$ is properly edge-colored,
    we have that  $d_{H_j'}(v) \ge d_{H_j}(v) - m \ge a$, so \ref{E1} holds.
    For every pair of disjoint sets $A$ and $B$ with orders at least $a$ and $b$, respectively,
    since $p b = C \sqrt{n} \log{n}/3 > 2\sqrt{n} \ge 2m$, 
    \begin{equation*}
      e_{H'_j}(A, B) \ge e_{H_j}(A, B) - m |A| \ge (1 - \varepsilon)p|A||B| - m |A| = |A|((1 - \varepsilon)p|B| - m) > 0.
    \end{equation*}
    Hence, we have established that \ref{E1} and \ref{E3} from the definition of an $(a,b)$-expander
    hold in $H_j'$. As was mentioned in the definition of an $(a,b)$-expander, \ref{E3} implies \ref{E2},
    so this completes the proof of the claim.
  \end{proof}

  Because $\gamma \delta = 4/n$ and $\gamma^2 = o(1/n)$,
  we have that $3 \gamma \delta - \gamma^2/2 \ge 1/n$, so we can apply Lemma~\ref{path_builder} to
  form a rainbow path forest $\mathcal{P} = \{P_1, \dotsc, P_r\}$ such that 
  $|V(\mathcal{P})| \ge (1 - 4 \delta)n$ and $r \le \gamma n$.

  We now apply the following algorithm to $\mathcal{P}$.
  \begin{itemize}
    \item If $|P_1| \ge n - 4 \delta n - a$ or one of $H_1$, $H_2$, or $H_3$ is not an $(a,b)$-expander, then terminate.
    \item Otherwise, we let $\mathcal{P}'$ and $e_1,e_2$ and $e_3$ be as in the statement of Lemma~\ref{path_builder}.
    \item For each $j \in [3]$, remove the color class corresponding to $e_j$ from $H_j$ and then repeat with $\mathcal{P} = \mathcal{P}'$.
  \end{itemize}

  Note that at most $(r-1) \log n < \sqrt{n}$ iterations of the algorithm will execute,
  since each of the $r-1$ paths $\{P_2, \dotsc, P_r\}$
  can be used to extend $P_1$ at most $\log n$ times. To see this,
  observe that every time such a path $P_i$ is used to extend $P_1$, 
  at least half of the remaining vertices in $P_i$ are removed.
  Therefore, by Claim~\ref{expander}, the algorithm must terminate with $|P_1| \ge n - 4 \delta n - a$. 

  After the algorithm terminates, 
  we have $(a,b)$-expanders $H_1$, $H_2$ and $H_3$ and a rainbow path $P_1$ on at least $n - 4 \delta n - a$ vertices
  such that the edges of $H_1$, $H_2$, $H_3$ and $P_1$ are colored with disjoint sets of colors.
  We can now use a procedure similar to the one in the proof of Lemma~\ref{path_builder} 
  to form a rainbow cycle of length at least $n - 4 \delta n - 3a$ and this will complete the proof.
  Let $v_1,\dotsc, v_m$ be the vertices of $P_1$ in the order they appear on the path.
  Let $A_1 = \{v_1, \dotsc, v_a\}$ be the first $a$ vertices on $P_1$ and
  let $A_2 = \{v_{m - (a-1)}, \dotsc, v_m\}$ be the last $a$ vertices on $P_1$.
  Assume $E_{H_1}(A_1, A_2) = \emptyset$, since otherwise we have the desired cycle.
  Let 
  \begin{equation*}
    B = \{v_j : v_{j+1} \in N_{H_1}(A_1) \cap (V(P_1) \setminus A_1)\},
  \end{equation*}
  and note that $A_2$ and $B$ are disjoint, and, because $H_1$ is an $(a,b)$-expander,
  \begin{equation*}
    |B| \ge |N_{H_1}(A_1)| - |V(G) \setminus V(P_1)| - |A_1| \ge (n - a - b) - (4 \delta n + a)  - a \ge b.
  \end{equation*}
  Therefore, there exists $v_j \in B$ and $v_k \in A_2$ such that $v_kv_j \in E_{H_2}(A_2, B)$.
  Recall that there exists $v_i \in A_1$ such that $v_iv_{j+1} \in E(H_1)$ and note that  $i < j < k$.
  Because $v_i,v_{j+1},\dotsc,v_k,v_j,v_{j-1},\dotsc,v_i$ is a cycle that contains all of the vertices
  $v_i, \dotsc, v_k$, we have the desired cycle.
\end{proof}
\setcounter{theorem}{\savetheoremcounter}

\section{Spanning Rainbow Path Forest with few paths}

In this section we prove Theorem~\ref{rpf}.  
In fact, we prove the following more general result which implies the theorem.
\begin{theorem}\label{rpf_degree_version}
  There exists a constant $C$ such that for every $n$ and for all $\delta=\delta(n) > 0$ the following holds.
  If $G$ is a properly edge-colored graph on $n$ vertices and $\delta(G) \ge \left(1 - \delta\right)n$,
  then $G$ contains a spanning rainbow path forest $\mathcal{P}$ with at most $C(\log n)^2 + 3 \delta n$ paths.
\end{theorem}

We will need the following technical lemma. We defer its proof until 
after the proof of Theorem~\ref{rpf_degree_version}.
\begin{lemma}\label{sequence_bounds}
  Suppose that $0 < c \le 1$ and that $n_1, \dotsc, n_k$ is a sequence of strictly increasing positive integers
  such that for all $m \le j < \ell \le k$ 
  \begin{equation}\label{sj}
    n_j - n_{j-1} \ge \frac{n_\ell - n_j}{n_j} \left( \left( 1 + c \right) n_j - (2 n_\ell - n_{\ell-1}) \right).
  \end{equation}
  Then $k \le \left(\log_r n_k\right)^2 + 2 \log_r n_k + m + 1$ where $r = 1 + c/3$.
\end{lemma}

\edef\savetheoremcounter{\thetheorem}
\setcounterref{theorem}{rpf_degree_version}
\setcounter{claimcounter}{0}
\begin{proof}[Proof of Theorem~\ref{rpf_degree_version}]
Let $G$ be a properly edge-colored graph on $n$ vertices 
such that $\delta(G) \ge (1 - \delta)n$.
For a rainbow path forest $F$, let $p(F)$ be the number of
paths in $F$ and let $A(F)$ be the set of endpoints of paths in $F$.
For two rainbow path forests, $F$ and $F'$,
we say that $F'$ is obtained from $F$ by a \textit{swap} 
if there exists an edge $e$ in $G$ that is incident to endpoints of distinct paths in $F$
such that either $F' = F + e$ and there are no edges in $E(F)$ given the same color as $e$, or
$F' = F + e - e'$ where $e' \in E(F)$ such that $e'$ and $e$ are given the same color.
Note that when $F'$ is obtained from $F$ by a swap, there is a unique color, say $\alpha$,
that is used on the edges in $E(F') \triangle E(F)$.
We call $\alpha$ the color \textit{associated} with the swap.

Let $\mathcal{S}(F)$ be the set of rainbow path forests $F'$ that can
be obtained from $F$ by a sequence of swaps,
i.e.,\ $F' \in \mathcal{S}(F)$ if there exists a sequence of rainbow path forests
$F = F_1, F_2, \dotsc, F_m = F'$ such that, for every $i \in [m-1]$, 
$F_{i+1}$ is obtained from $F_i$ by a swap.
Note that $p(F) \ge p(F')$ for all $F' \in \mathcal{S}(F)$.
If, for all $F' \in \mathcal{S}(F)$, we have that $p(F')=p(F)$ then we say that $F$ is \textit{swap-optimal}.
Let $C(\mathcal{S}(F))$ contain the set of colors $\alpha$ such that 
$\alpha$ is the color associated with a swap between two forests
in $\mathcal{S}(F)$. 
For every collection of path forests $\mathcal{F}$, 
let $A(\mathcal{F}) = \bigcup_{F \in \mathcal{F}} A(F)$ be the
set of vertices $x$ for which there exists at least one path 
forest in $\mathcal{F}$ in which $x$ is an endpoint of a path.

Let $k$ be the minimum number of paths in a spanning rainbow path forest of $G$
and fix $F_k$ a rainbow spanning path forest with $k$ paths.
Note that $F_k$ is swap-optimal.
We use the following iterative procedure to select 
swap-optimal forests $F_k \supseteq F_{k-1} \supseteq \dotsm \supseteq F_{1}$.
Suppose that, for $j \ge 2$,  $F_k, F_{k-1}, \dotsc, F_j$, have been 
selected so that for all $j \le \ell \le k$, $F_\ell$ is swap-optimal and $p(F_\ell) = \ell$.
To select the forest $F_{j-1}$, 
we first define a forest $F^x_{j-1}$ and path $P^x_{j-1}$ for every $x \in A(\mathcal{S}(F_j))$.
To this end, let $x \in A(\mathcal{S}(F_j))$, and
\begin{enumerate}[label=(\roman*), noitemsep]
  \item\label{Pxj1}  pick $F \in \mathcal{S}(F_j)$ such that $x$ is an endpoint of a path $P$ in $F$, and,
  \item\label{Pxj2}  subject to \ref{Pxj1}, the path $P$ in $F$ containing $x$ is as short as possible.
\end{enumerate}
Note that for every $P$ and $F$ selected in this way, $F - P$ is swap-optimal.
To see this, first note that \ref{Pxj2} implies that the colors in $C(\mathcal{S}(F - P))$
are not used on the edges of $P$. 
But then, for every $F' \in \mathcal{S}(F - P)$, 
the forest $P + F'$ is rainbow, so $P + F' \in \mathcal{S}(F)$.
This implies that $p(P + F') = p(F) = j$, so $p(F') = p(F - P) = j-1$,
which further implies that $F - P$ is swap-optimal.
Define $P^x_{j-1} = P$ and $F^x_{j-1} = F - P^x_{j-1}$.
To complete the procedure for constructing $F_{j-1}$, 
pick $x \in A(\mathcal{S}(F_{j})$ so that $|A(\mathcal{S}(F^x_{j-1})|$ is as small as possible,
and then let $F_{j-1} = F^x_{j-1}$.

For every $j \in [k]$, 
define $A_j = A(\mathcal{S}(F_j))$, $n_j = |A_j|$, $C_j = C(\mathcal{S}(F_j))$ and $G_j$ to be the graph with
vertex set $V(G)$ that contains only the edges of $G$ that are assigned a color from $C_j$.
Define $d_j(x) = d_{G_j}(x)$ for  $x \in V(G)$,
and, for  $U \subseteq V(G)$, let $d_j(x, U) = |N_{G_j}(x) \cap U|$. 
Similarly, for disjoint vertex subsets $A$ and $B$, we let 
$E_j(A, B) = E_{G_j}(A, B)$ and $e_j(A,B) = e_{G_j}(A, B)$.

The following claim summarizes some of the important facts implied by this construction.
\begin{claim}\label{construction}
  For every $1 \le j \le k$, we have that $F_j$ is swap-optimal. 
  For every $2 \le j \le k$, every $x \in A_j$ and every $F \in \mathcal{S}(F^x_{j-1})$,
  we have that $P^x_{j-1} + F \in \mathcal{S}(F_j)$, so
  $A(\mathcal{S}(F^x_{j-1})) \subseteq A_j$. 
  We also have that $|A(\mathcal{S}(F^x_{j-1}))| \ge n_{j-1}$.
\end{claim}

\begin{claim}\label{CF}\label{size_of_Cj}
  For every $1 \le j \le k$, we have that $\frac{1}{2} n_j - j \le |C_j| \le n_j - j$.
\end{claim}
\begin{proof}
  Let $H$ be the subforest of $F_j$ 
  created by first removing from $F_j$ all edges that were not assigned a color from $C_j$ and 
  then removing all isolated vertices that are not an endpoint of a path in $F_j$.
  Recall that for $v \in V(F_j)$, we have that $v \in A_j$ if and only if 
  there exists a path forest in $\mathcal{S}(F_j)$ in which $v$ is an endpoint of a path.
  Therefore, $A_j$ contains all of the endpoints of paths in $F_j$ and these endpoints are
  also in $V(H)$.
  Now consider a vertex $v$ in $V(F_j)$ that is not the endpoint of a path in $F_j$.
  Then $v \in A_j$ if and only if one of the two edges incident to $v$ in $F_j$ is colored
  with a color from $C_j$.
  Note we have established that $A_j = V(H)$.

  Since $H$ is rainbow and $C_j$ is exactly the set of colors used on the edges of $H$,
  $|E(H)| = |C_j|$.
  Because $H$ is a path forest and $|V(H)| = |A_j| = n_j$, 
  we have that $|C_j| = |E(H)| = n_j - p(H)$.
  Therefore, to complete the proof, we only need to show that the number of paths
  in $H$, $p(H)$, is between $j$ and $\frac{1}{2}n_j + j$.
  The lower bound on $p(H)$ follows because $p(F_j) = j$ and $H \subseteq F_j$.
  The upper bound on $p(H)$ comes from the fact that the isolated vertices in $H$
  must be endpoints of paths in $F_j$. 
  Therefore, there are at most $2j$ isolated vertices in $H$.
  Since all of the paths in $H$ that are not isolated vertices must 
  contain at least $2$ vertices,
  $p(H) \le \frac{1}{2}(n_j - 2j) + 2j = \frac{1}{2}n_j + j$.
\end{proof}

\begin{claim}\label{degree_lemma}
  For every $j \in [k]$ and $x \in A_j$, we have that
  $d_j(x) \le |C_j|$ and $d_j(x, A_j) \ge n_{j-1} - \delta n$. 
\end{claim}
\begin{proof}
  The first inequality, $d_j(x) \le |C_j|$, 
  follows from the fact that $G_j$ is properly edge-colored and only uses colors from $C_j$.
  To establish the second inequality, $d_j(x, A_j) \ge n_{j-1} - \delta n$, we first note that, by Claim~\ref{construction},
  $A(\mathcal{S}(F^x_{j-1}))$ is a subset of $A_j$, 
  and that $|A(\mathcal{S}(F^x_{j-1}))| \ge n_{j-1}$.
  Therefore, since $\delta(G) \ge n - \delta n$, we will 
  prove the second inequality by showing that if
  $y \in A(\mathcal{S}(F^x_{j-1}))$ such that $xy \in E(G)$,
  then the edge $xy$ is assigned a color from $C_j$.

  To this end, let $y \in A(\mathcal{S}(F^x_{j-1})) \cap N_G(x)$
  and let $\alpha$ be the color assigned to $xy$.
  Recall that $x$ is one of the endpoints of $P^x_{j-1}$ and that,
  by the definition of $A(\mathcal{S}(F^x_{j-1}))$, 
  there exists $F \in \mathcal{S}(F^x_{j-1})$ in which $y$ is the endpoint of a path.
  By Claim~\ref{construction}, if $F' = P^x_{j-1} + F$, then $F' \in \mathcal{S}(F_j)$.
  Furthermore, because $F_j$ is swap-maximal, 
  there exists $e \in E(F')$ such that $e$ is assigned the color $\alpha$.
  Therefore, $F' + xy - e \in \mathcal{S}(F_j)$, which implies $\alpha \in C_j$. 
  Since $F'$ is obtained from $F$ by a swap and $\alpha$ is the color associated
  with this swap, we have that $\alpha \in C_j$.
  This implies the claim.
\end{proof}

\begin{claim}\label{edge_bounds}
  For every $2 \le j < \ell \le k$, 
  \begin{equation*}
    (n_\ell - n_j)\left(\frac{3}{2}n_j - 2n_\ell + n_{\ell-1} - \delta n\right) \le e_j(A_\ell \setminus A_j, A_j) \le 
    n_j(n_j - n_{j-1} - j + \delta n).
  \end{equation*}
\end{claim}
\begin{proof}
  Let $x \in A_\ell \setminus A_j$.
  First note that $d_j(x, A_j) \ge d_\ell(x, A_{j}) - |C_\ell \setminus C_j|$
  because the edges of $G$ and hence $G_\ell$ are properly colored.
  Because Claim~\ref{size_of_Cj} implies that 
  \begin{equation*}
    |C_\ell \setminus C_j| \le (n_\ell - \ell) - \left(\frac{1}{2} n_j - j\right),
  \end{equation*}
  we have that 
  \begin{equation}\label{comp1}
    d_j(x, A_j) \ge d_\ell(x, A_\ell) - |C_\ell \setminus C_j|
                \ge d_\ell(x, A_j) - (n_\ell - \ell) + \left(\frac{1}{2} n_j - j\right).
  \end{equation}
  We also have that,
  \begin{equation*}
    d_\ell(x, A_j) \ge d_\ell(x, A_\ell) - |A_\ell \setminus A_j| = d_\ell(x, A_\ell) - (n_\ell - n_j),
  \end{equation*}
  and then by Claim~\ref{degree_lemma}, 
  \begin{equation*}
    d_\ell(x, A_\ell) \ge n_{\ell-1} - \delta n,
  \end{equation*}
  which using \eqref{comp1} gives us that
  \begin{align*}
    d_j(x, A_j)                 &\ge  \left(\left(n_{\ell-1} - \delta n\right)  - (n_\ell - n_j)\right) - (n_\ell - \ell) + \left(\frac{1}{2}n_j - j\right)  \\
                &= \frac{3}{2} n_j + n_{\ell-1} - 2 n_\ell + \ell - j - \delta n 
                \ge \frac{3}{2} n_j - 2 n_\ell + n_{\ell-1} - \delta n.
  \end{align*}
  Summing up the above relation over all vertices in $A_\ell \setminus A_j$ gives the lower bound.

  For every $x \in A_j$, 
  by Claims~\ref{degree_lemma} and~\ref{size_of_Cj},
  $d_j(x, A_\ell) \le d_j(x) \le |C_j| \le n_j - j$.
  By Claim~\ref{degree_lemma}, we also have that
  $d_j(x, A_j) \ge n_{j-1} - \delta n$.
  Therefore, 
  \begin{equation*}
    d_j(x, A_\ell \setminus A_j) = 
    d_j(x, A_\ell) - d_j(x, A_j) \le (n_j - j) - (n_{j-1} - \delta n).
  \end{equation*}
  Summing over all vertices in $A_j$ gives the upper bound.
\end{proof}

Let $m = \ceiling{ 3 \delta n }$, so, for every $m \le j \le k$, because $n_j \ge j \ge 3 \delta n$,
\begin{equation*}
  \frac{3}{2}n_j - 2n_\ell + n_{\ell-1} - \delta n \ge \frac{7}{6}n_j - (2n_\ell - n_{\ell-1})
\end{equation*}
and 
\begin{equation*}
  n_j - n_{j-1} - j + \delta n \le n_j - n_{j-1}.
\end{equation*}
Therefore, Claim~\ref{edge_bounds} implies that, for every $m \le j < \ell \le k$,
\begin{multline*}
  n_j - n_{j-1} \ge n_j - n_{j-1} - j + \delta n \ge  \\
  \frac{n_\ell - n_j}{n_j} \left(  \frac{3}{2}n_j - 2n_\ell + n_{\ell-1} - \delta n \right) \ge 
  \frac{n_\ell - n_j}{n_j} \left( \left( 1 + \frac{1}{6} \right) n_j - (2 n_\ell - n_{\ell-1}) \right).
\end{multline*}
We can then apply Lemma~\ref{sequence_bounds} to $n_1, \dotsc, n_k$ to deduce that
$k \le \left(\log_r n_k\right)^2 + 2\log_r n_k + m + 1$ where $r = 19/18$.
Hence, Theorem~\ref{rpf_degree_version} holds.
\end{proof}
\setcounter{theorem}{\savetheoremcounter}

\edef\savetheoremcounter{\thetheorem}
\setcounterref{theorem}{sequence_bounds}
\setcounter{claimcounter}{0}
\begin{proof}[Proof of Lemma~\ref{sequence_bounds}]
  \begin{claim}\label{c1}
    For all $m \le j < \ell \le k$,
    if $n_\ell \le r n_j$, then $n_\ell - n_j \le r^{-1}(n_\ell - n_{j-1})$.
  \end{claim}
  \begin{proof}
    We have that
    \begin{equation*}
      2n_\ell - n_{\ell - 1} \le 2n_\ell - n_j \le (2r - 1)n_j.
    \end{equation*}
    With the fact that $c = 3r - 3$, this implies that 
    \begin{equation*}
      (1 + c) n_j - (2 n_\ell - n_{\ell-1}) \ge (1 + c - (2r - 1))n_j = (r - 1)n_j.
    \end{equation*}
    Combining this with \eqref{sj} gives us that
    $(r - 1)(n_\ell - n_j) \le n_j - n_{j-1}$, so
    \begin{equation*}
      r(n_\ell - n_j) = n_\ell - n_j + (r-1)(n_\ell - n_j) \le n_\ell - n_j + n_j - n_{j-1} = n_\ell - n_{j-1},
    \end{equation*}
    which proves the claim.
  \end{proof}

  \begin{claim}\label{c2}
    For all $m \le j < \ell \le k$, if $n_\ell \le r n_j$, then $\ell - j < \log_{r}{n_j}$.
  \end{claim}
  \begin{proof}
    Note that if $n_\ell \le r n_j$ and $j \le  \ell - 1$, then for every $i$ such that $j \le i \le \ell - 1$, 
    we have that $n_\ell \le r n_j \le r n_i$.  Therefore,
    Claim~\ref{c1} implies $n_\ell - n_{i} \le r^{-1}(n_\ell - n_{i-1})$, which further implies that 
    \begin{equation*}
      1 \le n_\ell - n_{\ell - 1} \le r^{-(\ell - (j+1))}(n_\ell - n_{j}),
    \end{equation*}
    and $\ell - (j+1) \le \log_{r}(n_\ell - n_j)$.
    Using this and the fact that our assumption $n_\ell \le r n_j$ implies that $n_\ell - n_j \le (r - 1)n_j$ we have that
    \begin{equation*}
      \ell - (j + 1) \le \log_{r}(n_\ell - n_j) \le \log_r((r - 1)n_j).
    \end{equation*}
    Therefore, because $1 < r \le 4/3$ implies that $r - 1 < r^{-1}$, i.e., $\log_r(r-1)<-1$, we have
    \begin{equation*}
      \log_r{n_j} \ge \ell - j - 1 - \log_r (r-1) > \ell - j,
    \end{equation*}
    and the claim holds.
  \end{proof}

  Let $t = \ceiling{\log_r n_k}$ and $s = \floor{\frac{k - m}{t}}$.
  Note that, for every $1 \le i \le s$, 
  \begin{equation*}
    (it + m) - ((i-1)t + m) = t \ge \log_r n_k, 
  \end{equation*}
  and, therefore, Claim~\ref{c2} implies that
  \begin{equation*}
    n_{it + m} \ge r n_{(i-1)t + m},
  \end{equation*}
  so $n_{it + m} \ge r^{i} n_m$, and  
  \begin{equation*}
    n_k \ge n_{st + m} \ge r^s n_m \ge r^s.
  \end{equation*}  
  Therefore,
  \begin{equation*}
    \log_r{n_k} \ge s \ge \frac{k - m}{t} - 1 \ge \frac{k - m}{(\log_r n_k) + 1} - 1,
  \end{equation*}
  which implies 
  \begin{equation*}
    k \le (\log_r{n_k})^2 + 2\log_r{n_k} + m + 1. \qedhere
  \end{equation*}
\end{proof}
\setcounter{theorem}{\savetheoremcounter}


\begin{thebibliography}{10}
  \bibitem{aa} S. Akbari and  A. Alipour, Multicolored trees in complete graphs, J. Graph Theory, {\bf 54} (2007), 221--232.
  \bibitem{aemm} S. Akbari, O. Etesami, H. Mahini, and M. Mahmoody, On rainbow cycles in edge coloured complete
    graphs. Australas. J. Combin. {\bf 37} (2007), 33--42.
  \bibitem{aps} N. Alon, A. Pokrovskiy and B. Sudakov, Random subgraphs of properly edge-coloured complete graphs
    and long rainbow cycles, \textit{arXiv preprint arXiv:1608.07028}, 2016.
  \bibitem{A} L. Andersen, Hamilton circuits with many colours in properly edge-coloured complete graphs,
  Math. Scand., {\bf 64} (1989), 5--14.
  \bibitem{blm} J. Balogh, H. Liu and R. Montgomery, Rainbow spanning tree in properly edge-coloured complete graphs, 
    \textit{arXiv preprint arXiv:1704.07200}, 2017.
  \bibitem{chh} J. Carraher, S. Hartke, and P. Horn, Edge-disjoint rainbow spanning trees in complete graphs,
  European J. Combin., {\bf 57} (2016), 71--84.
  \bibitem{cl} H. Chen and X. Li, Long rainbow path in properly edge-coloured complete graphs, 
    \textit{arXiv preprint arXiv:1503.04516}, (2015).
  \bibitem{gm} H. Gebauer and F. Mousset, On Rainbow Cycles and Paths, 
    \textit{arXiv preprint arXiv:1207.0840}, (2012).
  \bibitem{gym} A. Gy\'arf\'as and M. Mhalla. Rainbow and orthogonal paths in factorizations of $K_n$, 
    J. Combin. Des. {\bf 18} (2010), 167-176.
  \bibitem{grss} A. Gy\'arf\'as, M. Ruszink\'o, B. S\'ark\"ozy, and R. Schelp. Long rainbow cycles in proper edge-colourings of
    complete graphs, Australas. J. Combin. {\bf 50} (2011), 45--53.
  \bibitem{hs} P. Hatami and P.W. Shor, A lower bound for the length of a partial transversal in a Latin square, 
  J.~Combin. Theory A, {\bf 115} (2008), 1103--1113.
  \bibitem{hn} P. Horn and L. Nelson, Many edge-disjoint rainbow spanning tress in general graphs,
    \textit{arXiv preprint arXiv:1704.00048}, 2017.
  \bibitem{ps} A. Pokrovskiy and  B. Sudakov, Linearly many rainbow trees in properly edge-coloured complete graphs, 
    \textit{arXiv preprint arXiv:1703.07301}, (2017).

\end{thebibliography}
\end{document}